\documentclass[12pt,reqno]{amsart}
\usepackage{amsmath,amsthm,amsfonts,amssymb,times}
\usepackage{verbatim}
\setbox0=\hbox{$+$}
\newdimen\plusheight
\plusheight=\ht0
\def\+{\;\lower\plusheight\hbox{$+$}\;}

\setbox0=\hbox{$-$}
\newdimen\minusheight
\minusheight=\ht0
\def\-{\;\lower\minusheight\hbox{$-$}\;}

\setbox0=\hbox{$\cdots$}
\newdimen\cdotsheight
\cdotsheight=\plusheight
\def\cds{\lower\cdotsheight\hbox{$\cdots$}}

\makeatletter
\def\leqalignno#1{\displ@y \tabskip\z@ plus\@ne fil
  \halign to\displaywidth{\hfil$\@lign\displaystyle{##}$\tabskip\z@skip
    &$\@lign\displaystyle{{}##}$\hfil\tabskip\z@ plus\@ne fil
    &\kern-\displaywidth\rlap{$\@lign\hbox{\rm##}$}\tabskip\displaywidth\crcr
    #1\crcr}}
\makeatother

\renewcommand{\i}{\infty}

\newcommand{\eb}{\begin{equation}}
\newcommand{\ee}{\end{equation}}

\newcommand{\df}{\dfrac}

\newcommand{\csch}{\operatorname{csch}}

\renewcommand{\Re}{\operatorname{Re}}

\newcommand{\dn}{\operatorname{dn}}
\newcommand{\sn}{\operatorname{sn}}
\newcommand{\cn}{\operatorname{cn}}

\renewcommand{\Re}{\text{Re}}

\renewcommand{\(}{\left\(}
\renewcommand{\)}{\right\)}
\renewcommand{\[}{\left\[}
\renewcommand{\]}{\right\]}
\renewcommand{\i}{\infty}

\numberwithin{equation}{section}
 \theoremstyle{plain}
\newtheorem{theorem}{Theorem}[section]

\begin{document}

\title[Two-dimensional series evaluations via Ramanujan and Jacobi]
{Two-dimensional series evaluations via the elliptic functions of
Ramanujan and Jacobi}

\author{Bruce C.~Berndt$^1$}
\thanks{1 Research supported by National Security Agency grant
H98230-11-1-0200}
\address{Department of Mathematics, University of Illinois, 1409
West Green St., Urbana, IL 61801, USA}

 \author{George Lamb}
 \address{2942 Ave.~del Conquistador, Tucson, AZ 85749-9304, USA}

 \author{Mathew Rogers$^2$}
 \thanks{2 Research supported by National Science Foundation
grant DMS-0803107.}
\address{Department of Mathematics, University of Illinois, 1409
West Green St., Urbana, IL 61801, USA}

\maketitle

\noindent{\footnotesize{\bf Abstract}. We evaluate in closed form,
for the first time,  certain classes of double series, which are
remindful of lattice sums.  Elliptic functions, singular moduli,
class invariants, and the Rogers--Ramanujan continued fraction
play central roles in our evaluations.}

\section{Introduction}

In this paper we establish elementary evaluations of certain
$2$-dimensional infinite series. For example,
\begin{equation}\label{Lamb's formula}
\begin{split}
\sum_{n=-\infty}^{\infty}\sum_{m=-\infty}^{\infty}\frac{(-1)^{m+n}}{(5m)^2+(5n+1)^2}
=&-\frac{\pi}{5\sqrt{5}}\log\left(\sqrt{5} + 1 - \sqrt{5 + 2\sqrt{5}}\right)\\
   &+\frac{\pi}{25}\log\left(11 + 5\sqrt{5}\right),
\end{split}
\end{equation}
which is a problem submitted to the \emph{American Mathematical Monthly}
\cite{zm}.
The algebraic numbers on the right-hand side of \eqref{Lamb's
formula} arise from special values of the Rogers--Ramanujan
continued fraction.  In general, elementary evaluations are quite
rare for higher-dimensional lattice-type sums. For instance, the
third author has examined both double and quadruple sums in
connection with Mahler measures of elliptic curves; those sums
typically reduce to values of hypergeometric functions
\cite{RGlattice}, \cite{RZ}, \cite{RZ2}.  The most famous
higher-dimensional sum is the Madelung constant from
crystallography  \cite{Bor}, \cite{Cr}, \cite{FG}, \cite{Zu}. It
is highly unlikely that Madelung's constant possesses an
evaluation in closed form.

We produce many additional results along the lines of
\eqref{Lamb's formula}.  In fact, we show that it is possible to
evaluate
\begin{equation}\label{Fabx definition}
F_{(a,b)}(x):=\sum_{n=-\infty}^{\infty}\sum_{m=-\infty}^{\infty}
\frac{(-1)^{m+n}}{(xm)^2+(a n+b)^2}
\end{equation}
for any positive rational value of $x$, and for many values of
$(a,b)\in\mathbb{N}^2$. Since the series is not absolutely
convergent, we will calculate the $n$-index of summation using
$\sum_{n}=\lim_{N\rightarrow\infty}\sum_{-N<n<N}$. When $a\le 6$,
results from Ramanujan's notebooks apply, and the values of the sums
can be deduced from classical results in theta functions and
$q$-series. When $a>6$, the situation is slightly more complicated.
In those cases we require the added hypothesis that
$a\in2\mathbb{Z}$, and then we use properties of Jacobian elliptic
functions.

\section{The main theorem}
Before proving our main theorem, we note that it is possible to
calculate $F_{(a,b)}(x)$ to high numerical precision with the
formula
\begin{equation}\label{Fabx rapidly converging formula}
F_{(a,b)}(x)=\frac{\pi}{x}\sum_{n=-\infty}^{\infty}\frac{(-1)^n
\csch\left(\frac{\pi(a n+b)}{x}\right)}{a n+b},
\end{equation}
which can be established by substituting the partial fractions
decomposition for $\csch(z)$ \cite[p.~28, Entry 1.217]{gr} in
\eqref{Fabx definition}. Formula \eqref{Fabx rapidly converging
formula} provides an easy way to numerically verify results such
as \eqref{Lamb's formula} and \eqref{F511 formula}. For instance,
we calculated $F_{(5,1)}(1)$ to more than $100$ decimal places
 by summing over $-15\le n\le 15$.

\begin{theorem}\label{MT} Suppose that $a$ and $b$ are integers with $a\ge 2$, $(a,b)=1$, and
assume that $\textup{Re}(x)>0$.  Then
\begin{equation}\label{Fabx main formula}
F_{(a,b)}(x)=-\frac{2\pi}{a
x}\sum_{j=0}^{a-1}\omega^{-(2j+1)
b}\log\left(\prod_{m=0}^{\infty}\left(1-\omega^{2j+1}
q^{2m+1}\right)\left(1-\omega^{-2j-1}q^{2m+1}\right)\right),
\end{equation}
where $\omega=e^{\pi i/a}$ and $q=e^{-\pi /x}$.
\end{theorem}

\begin{proof} Suppose that $N$ is a positive integer.  Then using
the transformation formula for the theta function
$\varphi(-q)=\sum_{n=-\i}^{\i}(-1)^nq^{n^2}$ \cite[p.~43, Entry
27(ii)]{III} and inverting the order of summation and integration
twice by absolute convergence, we find that
\begin{align*}
\sum_{\substack{m\in\mathbb{Z}\\-N< n < N}}\frac{(-1)^{m+n}}{(x
m)^2+(a n+b)^2}=&\pi\sum_{-N<n< N}(-1)^n\int_{0}^{\infty}e^{-\pi
(a n+b)^2
u}\left(\sum_{m\in\mathbb{Z}}(-1)^m e^{-\pi m^2 x^2 u}\right)  du\\
=&\pi\sum_{-N< n< N}(-1)^n\int_{0}^{\infty}e^{-\pi (a n+b)^2
u}\left(\frac{2}{x\sqrt{u}}\sum_{m=0}^{\infty}e^{-\frac{\pi (2m+1)^2}{4 x^2 u}} \right) du\\
=&\frac{2\pi}{x}\sum_{m=0}^{\infty}\sum_{-N< n<
N}(-1)^n\int_{0}^{\infty}e^{-\pi (a n+b)^2 u}e^{-\frac{\pi
(2m+1)^2}{4 x^2 u}} \frac{du}{\sqrt{u}}.
\end{align*}
The substitution of a standard $K$-Bessel function integral
\cite[p.~384, formula 3.471, no.~9]{gr}
\begin{equation*}
\int_{0}^{\infty}e^{-\pi\left(A^2 u+B^2/u\right)}\frac{d
u}{\sqrt{u}}=\frac{e^{-2 \pi |A| |B|}}{|A|}
\end{equation*}
leads to
\begin{equation}\label{main thm intermed}
\sum_{\substack{m\in\mathbb{Z}\\-N< n< N}}\frac{(-1)^{m+n}}{(x
m)^2+(a n+b)^2}=\frac{2\pi}{x}\sum_{m=0}^{\infty}\sum_{-N< n<
N}\frac{(-1)^n q^{(2m+1)|a n+b|}}{|a n+b|}.
\end{equation}
Notice that the condition $\Re(x)>0$  ensures convergence. Now
recall that $\omega=e^{\pi i/a}$. If $a\ge 2$, $(a,b)=1$, and
$|r|<1$, then
\begin{equation}\label{main thm log sum evaluation}
\sum_{n=-\infty}^{\infty}\frac{(-1)^n r^{|a n+b|}}{|a
n+b|}=-\frac{1}{a}\sum_{j=0}^{a-1}\omega^{-(2j+1)
b}\log\left(\left(1-\omega^{2j+1} r\right)\left(1-\omega^{-(2j+1)}
r\right)\right),
\end{equation}
which is easily verified by comparing Taylor series coefficients
in $r$. Notice that \eqref{main thm log sum evaluation} involves
an infinite series in $n$, whereas \eqref{main thm intermed}
imposes the restriction that $n\in(-N,N)$. If we divide the
left-hand side of \eqref{main thm log sum evaluation} into three
components, $-\infty<n\leq-N, -N<n<N, N\leq{n}<\infty$, and use a crude
error estimate to bound the terms where $n\ge N$ and $n\le -N$,
then we find that
\begin{equation}\label{main thm last step}
\begin{split}
\sum_{-N<n<N}\frac{(-1)^n r^{|a n+b|}}{|a
n+b|}=&-\frac{1}{a}\sum_{j=0}^{a-1}\omega^{-(2j+1)
b}\log\left(\left(1-\omega^{2j+1} r\right)\left(1-\omega^{-(2j+1)}
r\right)\right)\\
&+O\left(\frac{r^N}{(1-r)N}\right).
\end{split}
\end{equation}
Substituting \eqref{main thm last step} into \eqref{main thm
intermed} leads to equation \eqref{Fabx main formula} plus an
error term. The error term can easily be seen to approach zero as
$N\rightarrow \infty$.
\end{proof}

\section{Simplification for $a\in\{3,4,5,6\}$}

Although \eqref{Fabx main formula} was not difficult to prove, the
deduction of results such as \eqref{Lamb's formula} from
\eqref{Fabx main formula} is usually more difficult. In this
section we
 examine the cases where $a\in \{3,4,5,6\}$. In these instances
we can assume that $b=1$ without loss of generality, because a
standard symmetry (e.g. $n\rightarrow -n$) can be used to recover
the other possible values of $F_{(a,b)}(x)$. The same symmetry
immediately implies that $F_{(2,1)}(x)=0$.

Let us briefly recall  the $q$-series notation
$$  (a;q)_{\infty} := \prod_{j=0}^{\infty}(1-aq^j), \qquad |q| < 1. $$
Following Ramanujan's  notation for theta functions, define
\begin{align*}
\varphi(q)&=\sum_{n=-\infty}^{\infty}q^{n^2},&\psi(q)&=\sum_{n=0}^{\infty}q^{\frac{n(n+1)}{2}},\\
\chi(q)&=(-q;q^2)_{\infty},&f(-q)&=(q;q)_{\infty}.
\end{align*}
We need the famous Jacobi triple product identity \cite[p.~35,
Entry 19]{III} for Ramanujan's general theta function $f(a,b)$,
and they are given by
\begin{equation}\label{fabjtp}
f(a,b) := \sum_{n=-\infty}^{\infty}a^{n(n+1)/2}b^{n(n-1)/2}
=(-a;ab)_{\infty}(-b;ab)_{\infty}(ab;ab)_{\infty}, \qquad |ab|<1.
\end{equation}
We note the easily proved evaluation \cite[p.~34, Entry
18(ii)]{III}
\begin{equation}\label{zero}
f(-1,q)=0, \qquad |q|<1.
\end{equation}
 The Rogers--Ramanujan
continued fraction
\begin{equation*}
R(q):=\df{q^{1/5}}{1}\+\df{q}{1}\+\df{q^2}{1}\+\df{q^3}{1}\+\cds,
\qquad |q|<1,
\end{equation*}
 also plays an important role in this paper.
We  always assume that $q^{1/5}$ takes the principal value, so
that $R(-q)$ assumes a real value if $q\in(0,1)$.

\begin{theorem}\label{thm for a<6} Suppose that $q=e^{-\pi/x}$. Let $u=-q^{1/3}\chi(q)/\chi^3(q^3)$,
 $\alpha=1-\varphi^4(-q^8)/\varphi^4(q^8)$, and $\mu=R(-q)R^2(q^2)$.
 Then
\begin{align}
F_{(3,1)}(x)
=&-\frac{2\pi}{9x}\log\left(\frac{1+u^3}{1-8u^3}\right),\label{F13x formula}\\
F_{(4,1)}(x)
=&-\frac{\pi}{\sqrt{2}x}\log\left(\frac{1-\sqrt[8]{\alpha}}{1+\sqrt[8]{\alpha}}\right),
\label{F14x formula}\\
F_{(5,1)}(x)=&-\frac{\pi}{5\sqrt{5}x}
  \log\left(\frac{2 - \mu + 18 \mu^2 + \mu^3 + 2 \mu^4 + 5 \sqrt{5} (\mu +
  \mu^3)}{
   2 - \mu + 18 \mu^2 + \mu^3 + 2 \mu^4 - 5 \sqrt{5}( \mu +\mu^3)}\right)\label{F15x formula}\\
    &-
 \frac{\pi}{5 x}\log\left(\frac{1 + \mu - \mu^2}{1 -4 \mu -\mu^2}\right),\notag\\
F_{(6,1)}(x)=&-\frac{\pi}{\sqrt{3}
x}\log\left(\frac{\varphi(-q^4)-3\varphi(-q^{36})+2\sqrt{3} q
f(-q^{24})}{\varphi(-q^4)-3\varphi(-q^{36})-2\sqrt{3} q
f(-q^{24})}\right).\label{F16x formula}
\end{align}
\end{theorem}
\begin{proof} We begin by proving \eqref{F13x formula}.
If we set $(a,b)=(3,1)$, then \eqref{Fabx main formula}
immediately reduces to
\begin{align}
F_{(3,1)}(x)&=-\frac{2\pi}{3x}\sum_{j=0}^{2}\cos\left(\frac{\pi(2j+1)}{3}\right)
\log\prod_{m=0}^{\infty}\left(1-2\cos\left(\frac{\pi
(2j+1)}{3}\right)q^{2m+1}+q^{4m+2}\right)\notag\\
&=-\frac{2\pi}{3x}\log\prod_{m=0}^{\infty}\frac{1-q^{2m+1}+q^{4m+2}}{1+2q^{2m+1}+q^{4m+2}}
=-\frac{2\pi}{3x}\log\frac{\chi(q^3)}{\chi^3(q)}.\label{f31 in
terms of chi}
\end{align}
To finish the calculation, let us briefly take
$\alpha=1-\varphi^4(-q)/\varphi^4(q)$ and
$\beta=1-\varphi^4(-q^3)/\varphi^4(q^3)$.  Then by the inversion
formula for $q^{-1/24}\chi(q)$ \cite[p.~124, Entry 12(v)]{III}, we
have
\begin{align*}
\frac{\chi^3(q)}{\chi(q^3)}=2^{1/3}\left(\frac{\beta(1-\beta)}{\alpha^3(1-\alpha)^3}\right)^{1/24},
&&\frac{\chi^3(q^3)}{q^{1/3}\chi(q)}=2^{1/3}\left(\frac{\alpha(1-\alpha)}{\beta^3(1-\beta)^3}\right)^{1/24}.
\end{align*}
It is known that $\alpha$ and $\beta$ admit birational
parametrizations $\alpha=p(2+p)^3/(1+2p)^3$ and
$\beta=p^3(2+p)/(1+2p)$ \cite[~p.~230, Entry 5(vi)]{III}. Thus we
obtain
\begin{align*}
\frac{\chi(q^3)}{\chi^3(q)}=\left(\frac{(1-p)(2+p)}{2(1+2p)^2}\right)^{1/3},
&&\frac{q^{1/3}\chi(q)}{\chi^3(q^3)}=\left(\frac{p(1+p)}{2}\right)^{1/3}.
\end{align*}
Finally, it is easy to see that if
$u:=-q^{1/3}\chi(q)/\chi^3(q^3)$, then
\begin{equation*}
\frac{\chi(q^3)}{\chi^3(q)}=\left(\frac{1+u^3}{1-8u^3}\right)^{1/3}.
\end{equation*}
Substituting this last result into \eqref{f31 in terms of chi}
completes the proof of \eqref{F13x formula}.

Next we prove \eqref{F14x formula}.  Notice that if $(a,b)=(4,1)$,
then \eqref{Fabx main formula} becomes
\begin{align}
F_{(4,1)}(x)&=-\frac{2\pi}{4x}\sum_{j=0}^{3}\cos\left(\frac{\pi(2j+1)}{4}\right)
\log\prod_{m=0}^{\infty}\left(1-2\cos\left(\frac{\pi
(2j+1)}{4}\right)q^{2m+1}+q^{4m+2}\right)\notag\\
&=-\frac{\pi}{\sqrt{2}x}\log\prod_{m=0}^{\infty}\frac{(1-\sqrt{2}q^{2m+1}+q^{4m+2})(1-q^{2m+2})}
{(1+\sqrt{2}q^{2m+1}+q^{4m+2})(1-q^{2m+2})}.\label{F14nsimp}
\end{align}
Letting $\omega=e^{\pi i/4}$ and using the Jacobi triple product
identity \eqref{fabjtp}, we find that the denominator on the far
right side of \eqref{F14nsimp} is equal to
\begin{equation}
\begin{split}\label{hey}
F(q):=\prod_{m=0}^{\infty}&\left(1+\sqrt{2}q^{2m+1}+q^{4m+2}\right)\left(1-q^{2m+2}\right)\\
&=(-\omega
q;q^2)_{\infty}(-\bar{\omega}q;q^2)_{\infty}(q^2;q^2)_{\infty}\\
&=\sum_{m=-\infty}^{\infty}\omega^n q^{n^2}\\
&=\sum_{n=-\infty}^{\infty}(-1)^n\left(q^{(4n)^2}+\omega
q^{(4n+1)^2}+i q^{(4n+2)^2}+\omega^3 q^{(4n+3)^2}\right).
\end{split}
\end{equation}
Because the initial infinite product is real-valued, the imaginary
terms above sum to 0.  Alternatively, this fact also follows from
\eqref{zero}.  Hence, from \eqref{hey},
\begin{align}
F(q)&=\varphi(-q^{16})+\frac{1}{\sqrt{2}}
\sum_{n=-\infty}^{\infty}(-1)^n\left(q^{(4n+1)^2}-q^{(4n+3)^2}\right)\notag\\
&=\varphi(-q^{16})+\sqrt{2}\sum_{n=0}^{\infty}(-1)^{\frac{n(n+1)}{2}}
q^{(2n+1)^2}\notag\\
&=\varphi(-q^{16})+\sqrt{2}q\psi(-q^8).\label{F14 intermed triple}
\end{align}
A similar argument provides a similar representation for the
numerator on the far right side of \eqref{F14nsimp}. Hence, using
\eqref{F14 intermed triple} and its aforementioned analogue in
\eqref{F14nsimp}, we are led to the closed form
\begin{equation*}
F_{(4,1)}(x)=-\frac{\pi}{\sqrt{2}
x}\log\frac{\varphi(-q^{16})-\sqrt{2}q\psi(-q^8)}{\varphi(-q^{16})+\sqrt{2}q\psi(-q^8)}.
\end{equation*}
  If we take
$\alpha=1-\varphi^4(-q^8)/\varphi^4(q^8)$ and $z=\varphi^2(q^8)$,
 this expression reduces to \eqref{F14x formula} after applying
\cite[~p.~122, Entry 10(iii)]{III} and \cite[~p.~123, Entry
11(ii)]{III}.

Next we prove \eqref{F15x formula}.  This case is substantially
more difficult than the previous two.  If $(a,b)=(5,1)$, then
\eqref{Fabx main formula} becomes
\begin{align}
F_{(5,1)}(x)&=-\frac{2\pi}{5x}\sum_{j=0}^{4}\cos\left(\frac{\pi(2j+1)}{5}\right)
\log\prod_{m=0}^{\infty}\left(1-2\cos\left(\frac{\pi
(2j+1)}{5}\right)q^{2m+1}+q^{4m+2}\right)\notag\\
&=-\frac{\pi}{5x}\log\frac{\chi(q^5)}{\chi^5(q)}-
\frac{\pi}{\sqrt{5}x}\log\prod_{\text{$m$
odd}}^{\infty}\frac{1-\beta q^{m}+q^{2m}}{1-\alpha
q^{m}+q^{2m}},\label{F15nsimp}
\end{align}
where $\beta=\frac{1+\sqrt{5}}{2}$ and
$\alpha=\frac{1-\sqrt{5}}{2}$.  The second equality in
\eqref{F15nsimp} was obtained by collecting terms of the form
$\sqrt{5}\log(X)$.  Now we use several entries from Ramanujan's
lost notebook.  By \cite[pp.~21--22, Entry 1.4.1, Eqs.~(1.4.3),
(1.4.4)]{geaI},
\begin{equation}\label{thm f15x eval even part}
\prod_{\text{$m$ odd}}\left(\frac{1-\beta q^m+q^{2m}}{1-\alpha
q^m+q^{2m}}\right)=\sqrt[5]{\frac{\left(1-\alpha^5
R^5(-q)\right)\left(1-\beta^5 R^5(q^2)\right)}{\left(1-\beta^5
R^5(-q)\right)\left(1-\alpha^5 R^5(q^2)\right)}}.
\end{equation}
By \cite[p.~33]{geaI}, we can parameterize $R^5(-q)$ and
$R^5(q^2)$ in terms of $\mu=R(-q)R^2(q^2)$ with the identities
\begin{align}\label{R in terms of mu}
R^5(-q)=\mu\left(\frac{1-\mu}{1+\mu}\right)^2,&&R^5(q^2)=\mu^2\left(\frac{1+\mu}{1-\mu}\right).
\end{align}
Therefore \eqref{thm f15x eval even part} becomes
\begin{equation}\label{thm f15x big log done}
\prod_{\text{$m$ odd}}\left(\frac{1-\beta q^m+q^{2m}}{1-\alpha
q^m+q^{2m}}\right)=\sqrt[5]{\frac{2 - \mu + 18 \mu^2 + \mu^3 + 2
\mu^4 + 5 \sqrt{5} (\mu +
  \mu^3)}{
   2 - \mu + 18 \mu^2 + \mu^3 + 2 \mu^4 - 5 \sqrt{5}( \mu +\mu^3)}}.
\end{equation}
Finally, notice that $1/\chi(-q)=(-q;q)_{\infty}$. After replacing
$q$ by $-q$ in \cite[~p.~37, Entry 1.8.5]{geaI} and simplifying,
we have
\begin{equation}\label{thm f15x little log done}
\frac{\chi\left(q^5\right)}{\chi^5\left(q\right)}=\frac{1+\mu-\mu^2}{1-4\mu-\mu^2}.
\end{equation}
Substituting \eqref{thm f15x big log done} and \eqref{thm f15x
little log done} into \eqref{F15nsimp} concludes the proof of
\eqref{F15x formula}.

The proof of \eqref{F16x formula} is similar to the proof of
\eqref{F14x formula}, and we leave this calculation as an exercise
for the reader.  Note that the operative result
\begin{equation*}
\prod_{n=0}^{\infty}\left(\frac{1+\sqrt{3}q^{2n+1}+q^{4n+2}}{1-\sqrt{3}q^{2n+1}+q^{4n+2}}\right)
=\frac{-\varphi(-q^4)+3\varphi(-q^{36})+2\sqrt{3}q
f(-q^{24})}{-\varphi(-q^4)+3\varphi(-q^{36})-2\sqrt{3}q
f(-q^{24})}
\end{equation*}
follows easily from the Jacobi triple product identity
\eqref{fabjtp}.
\end{proof}

Now we derive some explicit examples from Theorem \ref{thm for
a<6}. All of our identities follow from well-known $q$-series
evaluations.  We begin with $F_{(3,1)}(x)$.  Notice by
\cite[~p.~95, Eq.~3.3.6]{geaI}, that $u=G(-q)$, where $G(q)$
denotes Ramanujan's cubic continued fraction defined by
$$
G(q):=\df{q^{1/3}}{1}\+\df{q+q^2}{1}\+\df{q^2+q^4}{1}\+\df{q^3+q^6}{1}\+\cds,
\qquad |q|<1.$$
 It follows that $F_{(3,1)}(x)$ can be evaluated by
using formulas for $G(-q)$. When $x=1$ we appeal to \cite[~p.~100,
Eq.~(3.4.1)]{geaI} to find that
$u=G(-e^{-\pi})=\frac{1-\sqrt{3}}{2}$, which yields
\begin{equation}
\sum_{n=-\infty}^{\infty}\sum_{m=-\infty}^{\infty}\frac{(-1)^{m+n}}{m^2+(3n+1)^2}
=\frac{2\pi}{9}\log\left(2(\sqrt{3}-1)\right).
\end{equation}
When $x=\frac{1}{\sqrt{5}}$ we appeal to \cite[~p.~101, Theorem
3.4.2]{geaI}. We have $u=G(-e^{-\pi\sqrt{5}})=\frac{(\sqrt{5} - 3)
(\sqrt{5} - \sqrt{3})}{4}$, and therefore
\begin{equation}
\sum_{n=-\infty}^{\infty}\sum_{m=-\infty}^{\infty}\frac{(-1)^{m+n}}{m^2+5(3n+1)^2}=\frac{\pi}{9\sqrt{5}}
\log\left(8(4-\sqrt{15})\right).
\end{equation}
It is possible to obtain many additional formulas for
$F_{(3,1)}(x)$, by applying formulas in
\cite[~pp.~100--105]{geaI}.

Now we examine $F_{(4,1)}(x)$.  This function is easy to examine,
because when $q=\exp(-\pi\sqrt{n})$, where $n\in \mathbb{Q}^+$,
the values of $\alpha_n$ are called singular moduli and can always
be calculated in terms of algebraic numbers \cite[p.~214]{cox}.
Their values have been extensively tabulated. For example, many
explicit evaluations of $\alpha_n$ can be found in
\cite[~pp.~281--306]{V}. Since $\alpha$ has an argument of
$q^8=e^{-8\pi/x}$, we write $\alpha=\alpha_{64/x^2}$. When $x=8$,
then $\alpha=\alpha_1=1/2$, and therefore
\begin{equation}
\sum_{n=-\infty}^{\infty}\sum_{m=-\infty}^{\infty}\frac{(-1)^{m+n}}{(8m)^2+(4n+1)^2}=\frac{\pi}{8\sqrt{2}}
\log\left(\frac{\sqrt[8]{2}+1}{\sqrt[8]{2}-1}\right).
\end{equation}
Similarly, when $x=4/\sqrt{7}$, we have \cite[p.~284]{V}
$\alpha=\alpha_{28}=(\sqrt{2}-1)^8\left(2\sqrt{2}-\sqrt{7}\right)^4$
\cite[~p.~284]{III}. Thus we obtain
\begin{equation}
\sum_{n=-\infty}^{\infty}\sum_{m=-\infty}^{\infty}\frac{(-1)^{m+n}}{(4m)^2+7(4n+1)^2}
=\frac{\pi}{4\sqrt{14}}\log\left(\frac{1+( \sqrt{2}-1) \sqrt{2
\sqrt{2} - \sqrt{7}}}{1-( \sqrt{2}-1) \sqrt{2 \sqrt{2} -
\sqrt{7}}}\right).
\end{equation}
There are many similar identities that follow from results in
\cite{V}, but they often tend to be very complicated. The majority
of the identities contain algebraic numbers involving nested
radicals.

We conclude this section by proving a pair of formulas for
$F_{(5,1)}(x)$.  By \eqref{F15x formula}, this requires
calculating the parameter $\mu$. In principle, these calculations
are straight-forward exercises. If the values of both $R(-q)$ and
$R(q^2)$ are known, then calculating $\mu=R(-q)R^2(q^2)$ is
trivial.  If only one of the values is known, then $\mu$ can be
calculated by solving \eqref{R in terms of mu}.  This second type
of calculation requires solving a cubic equation. If neither value
is known, then we can use \eqref{thm f15x little log done} to
calculate $\mu$.  In practice, we have only been able to identify
two instances where $\mu$ is reasonably simple.

We begin by setting $x=1$ in \eqref{F15x formula}. By
\cite[~pp.~57--58]{geaI},
\begin{equation*}\label{mu x=1 case}
\mu=R(-e^{-\pi})R^2(e^{-2\pi})=\frac{1}{8}(3 - \sqrt{5}) (7 + 3
\sqrt{5}) \left(4 + 2 \sqrt{5} -
        \sqrt{10 (5 + \sqrt{5})}\right).
\end{equation*}
 Substituting this last result into \eqref{F15x formula} and
simplifying with \texttt{Mathematica} leads to
\begin{equation}\label{F511 formula}
\begin{split}
\sum_{n=-\infty}^{\infty}\sum_{m=-\infty}^{\infty}&\frac{(-1)^{m+n}}{m^2+(5n+1)^2}\\
&=-\frac{\pi}{5\sqrt{5}}\log\left(-19+9\sqrt{5}-3\sqrt{85-38\sqrt{5}}\right)
+\frac{\pi}{5}\log\left(\sqrt{5}-1\right).
\end{split}
\end{equation}

Now we examine the more difficult case when $x=5$. This choice of
$x$ leads to \eqref{Lamb's formula} quoted in our Introduction. We
 calculate $\mu$ using \eqref{thm f15x little log done} and the values of class
 invariants $G_n$, which are defined by
 \begin{equation*}
 G_n:=2^{-1/4}q^{-1/24}\chi(q),
 \end{equation*}
 where $q=\exp(-\pi\sqrt{n})$.  Hence, using the fact that $G_n=G_{1/n}$ and the value of $G_{25}$
 \cite[~p.~190]{V}, we find that
\begin{equation*}
\frac{1+\mu-\mu^2}{1-4\mu-\mu^2}=\frac{\chi\left(e^{-\pi}\right)}{\chi^5\left(e^{-\pi/5}\right)}
=\frac{G_{1}}{2
G_{\frac{1}{25}}^5}=\frac{G_{1}}{2
G_{25}^5}=\frac{1}{2}\left(\frac{\sqrt{5}-1}{2}\right)^5.
\end{equation*}
Therefore $\mu$ is given by
\begin{equation*}
\mu= \frac{1}{4}\left(3 + \sqrt{5}\right)\left(-4 + 2 \sqrt{5} -
\sqrt{2 (25 - 11 \sqrt{5})}\right).
\end{equation*}
Substituting this last result into \eqref{F15x formula}, and then
simplifying nested radicals, we complete the proof of
\eqref{Lamb's formula}.

%
\section{Simplification for higher values}

In this section, we evaluate $F_{(a,b)}(x)$ when $a>6$.  In order
to simplify the calculations, we restrict our attention to cases
where $a\in2\mathbb{Z}$.  In these cases we can apply elementary
properties of Jacobian elliptic functions.  Let us briefly recall
that the elliptic functions $\sn(u)$, $\cn(u)$, and $\dn(u)$ are
doubly-periodic, meromorphic functions, which depend implicitly on
a parameter $\alpha=k^2$, where $k$ is called the elliptic
modulus. Their periods are integral multiples of $K$ and $i K'$,
where $K$ and $K'$ are complete elliptic integrals of the first
kind associated with the moduli $k$ and $k^{\prime}=\sqrt{1-k^2}$,
respectively.   For us, the representations in terms of
hypergeometric functions $_2F_1$ \cite[p.~102]{III}
\begin{align*}
K:=\frac{\pi}{2}{_2F_1}\left(\frac{1}{2},\frac{1}{2};1;\alpha\right),
&&K':=\frac{\pi}{2}{_2F_1}\left(\frac{1}{2},\frac{1}{2};1;1-\alpha\right)
\end{align*}
are employed in the sequel.  There is a well-known inverse
relation between the modulus, and the elliptic nome $q$ given by
\cite[p.~102]{III}
\begin{align*}
\alpha=1-\frac{\varphi^4(-q)}{\varphi^4(q)},
&&q=\exp\left(-\pi\frac{{_2F_1}\left(\frac{1}{2},\frac{1}{2};1;1-\alpha\right)}{{_2F_1}\left(\frac{1}{2},
\frac{1}{2};1;\alpha\right)}\right).
\end{align*}

\begin{theorem}\label{Jacobi Theorem}Suppose that $\alpha\in(0,1)$, and let
\begin{equation}\label{x in terms of alpha}
x=\frac{{_2F_1}\left(\frac{1}{2},\frac{1}{2};1;\alpha\right)}
{{_2F_1}\left(\frac{1}{2},\frac{1}{2};1;1-\alpha\right)}.
\end{equation}
Assume that $a\in\{2,4,6\dots\}$, $b\in\mathbb{Z}$, and $(a,b)=1$.
Then
\begin{equation}\label{Dn summation formula}
\begin{split}
F_{(a,b)}(x) &=\frac{\pi}{a
x}\sum_{j=0}^{a-1}\cos\left(\frac{\pi(2j+1)
b}{a}\right)\log\left(\dn\left(\frac{(2 j + 1)K}{a}\right)\right).
\end{split}
\end{equation}
\end{theorem}

\begin{proof}  We have already established that \eqref{Fabx main formula}
is true if $x\in(0,\infty)$. The identity remains valid if we let
$1/x\rightarrow1/x+i$.  This substitution has the effect of
sending $q\rightarrow -q$. Combining the two identities and
performing a great deal of simplification leads to
\begin{equation*}
\begin{split}
\sum_{n=-\infty}^{\infty}\sum_{m=-\infty}^{\infty}&\frac{(-1)^{m+n}(1-(-1)^{a
n+b})}{(x m)^2+(a
n+b)^2}\\
&=-\frac{2\pi}{a x}\sum_{j=0}^{a-1}\omega^{-(2j+1)
b}\log\left(\prod_{m=0}^{\infty}\frac{\left(1-\omega^{2j+1}
q^{2m+1}\right)\left(1-\omega^{-2j-1}q^{2m+1}\right)}{\left(1+\omega^{2j+1}
q^{2m+1}\right)\left(1+\omega^{-2j-1}q^{2m+1}\right)}\right).
\end{split}
\end{equation*}
Taking note of \eqref{x in terms of alpha}, and then recalling the
product representation for $\dn(u)$ \cite[~p.~918]{AS}, we find
that the last expression transforms into
\begin{equation*}
\begin{split}
\sum_{n=-\infty}^{\infty}\sum_{m=-\infty}^{\infty}&\frac{(-1)^{m+n}(1-(-1)^{a
n+b})}{(x m)^2+(a
n+b)^2}\\
&=\frac{2\pi}{a x}\sum_{j=0}^{a-1}\omega^{-(2j+1)
b}\log\left((1-\alpha)^{1/4}\dn\left(\frac{(2j+1)K}{a}\right)\right).
\end{split}
\end{equation*}
If we assume that $a$ is even and $b$ is odd, then the left-hand
side of the identity equals $2F_{(a,b)}(x)$.  We can recover
\eqref{Dn summation formula} by noting that $\sum_{j}
\omega^{-(2j+1) b}=0$ whenever $(a,b)=1$.
\end{proof}

In order to provide an application of  \eqref{Dn summation
formula}, we evaluate $F_{(10,1)}(1)$ explicitly.  While we
restrict our attention to this single example, the method we
describe extends to many additional values of $F_{(a,b)}(x)$.
 Let us recall that $\dn(u)$ has real period $2K$. If we use the
symmetries \cite[p.~500]{ww}
\begin{align}\label{d}
\dn(2K-u)=\dn(u),&& \dn(K-u)=\frac{\sqrt{1-\alpha}}{\dn(u)},
\end{align}
 then \eqref{Dn summation formula} reduces to an
expression involving $[\frac{a-1}{4}]$ elliptic functions. When
$(a,b)=(10,1)$, we have
\begin{equation}\label{10n+1 example}
\begin{split}
\frac{5x}{\pi}F_{(10,1)}(x)=&\sqrt{\frac{5 -\sqrt{5}}{2}}
\log\left(\dn\left(\frac{K}{10}\right)\right)- \sqrt{\frac{5
+\sqrt{5}}{2}} \log\left(\dn\left(\frac{3
K}{10}\right)\right) \\
&+
 \frac{\sqrt{5 - 2\sqrt{5}}}{4} \log(1 - \alpha).
\end{split}
\end{equation}
By  \eqref{x in terms of alpha}, it is possible to calculate
$\alpha$ whenever $x^2\in\mathbb{Q}^{+}$ and $x>0$. It just
remains to compute the values of the elliptic functions.

Notice that $\dn(r K/s)$ is an algebraic function of $\alpha$ if
$(r,s)\in\mathbb{Z}^2$.  This is a consequence of the fact that
elliptic functions obey addition formulas \cite[~p.~574]{AS}.
Perhaps the easiest method for calculating values such as
$\dn(K/10)$ and $\dn(3K/10)$ is to generate polynomials (but not
necessarily minimal ones) which they satisfy, by iterating the
duplication formula for $\dn(z)$ \cite[~p.~574]{AS}. Let us recall
that
\begin{equation}\label{dup}
\dn(2z)=f(\dn(z)),
\end{equation}
where
\begin{equation}\label{f(x) definition}
f(x):=\frac{x^2+(x^2-1)(1+\frac{1}{\alpha}(x^2-1))}{x^2-(x^2-1)(1+\frac{1}{\alpha}(x^2-1))}.
\end{equation}
For brevity, we use the shorthand notation
\begin{align*}
d_j:=\dn\left(\frac{j K}{10}\right).
\end{align*}
Using \eqref{d} and \eqref{dup}, we can easily show that
\begin{align*}
d_{2j}&=f(d_j),\\
d_j&=d_{20-j},\\
d_{10-j}&=\frac{\sqrt{1-\alpha}}{d_j},\\
d_5&=\sqrt[4]{1-\alpha}.
\end{align*}
As a consequence of the elementary properties above, it is easy to
deduce that
\begin{equation}\label{d1 and d3 polynomial}
\begin{split}
f(d_1) f(f(f(d_1)))-\sqrt{1-\alpha}=0,\\
f(d_3) f(f(f(d_3)))-\sqrt{1-\alpha}=0.
\end{split}
\end{equation}
For instance, notice that $f(d_1)f(f(f(d_1)))=d_2
f(f(d_2))=d_2f(d_4)=d_2 d_8=\sqrt{1-\alpha}$. It follows
immediately that $\dn(K/10)$ and $\dn(3K/10)$ are conjugate zeros
of a function which is rational in $\mathbb{Q}(\sqrt{1-\alpha})$.
It is easy to extract a polynomial which $d_1$ and $d_3$ satisfy,
by considering only the numerator of
$f(x)f(f(f(x)))-\sqrt{1-\alpha}$.

Now we can finish the computation of $F_{(10,1)}(1)$. Equation
\eqref{x in terms of alpha} shows that $x=1$ when $\alpha=1/2$. As
a result, \eqref{10n+1 example} becomes
\begin{equation}\label{10n+1 examples}
\begin{split}
\frac{5}{\pi}\sum_{n=-\infty}^{\infty}\sum_{m=-\infty}^{\infty}\frac{(-1)^{n+m}}{m^2+(10
n+1)^2}=&\sqrt{\frac{5 -\sqrt{5}}{2}} \log\left(d_1\right)-
\sqrt{\frac{5 +\sqrt{5}}{2}} \log\left(d_3\right) \\
 &-
 \frac{\sqrt{5 - 2\sqrt{5}}}{4} \log 2,
\end{split}
\end{equation}
where
\begin{align*}
d_1=\dn(K/10)\approx0.9915\dots,&& d_3=\dn(3K/10)\approx
0.9309\dots.
\end{align*}
If we use \texttt{Mathematica} to expand \eqref{d1 and d3
polynomial}, then it is easy to see that $d_1$ and $d_3$ are
conjugate zeros of the irreducible polynomial
\begin{equation*}\label{10n+1 minimal polynomial}
\begin{split}
0=&1 + 32 X^2 - 1152 X^4 + 14528 X^6 - 103328 X^8 + 445056 X^{10}
-
  747008 X^{12}\\
 &- 5859584 X^{14} + 67132864 X^{16} - 404289024 X^{18} +
  1770485760 X^{20}\\
   &- 6097568768 X^{22} + 17124502016 X^{24} -
  40180561920 X^{26} + 80299532288 X^{28}\\
  & - 138787278848 X^{30} +
  209592829440 X^{32} - 277574557696 X^{34}\\
   &+ 321198129152 X^{36} -
  321444495360 X^{38} + 273992032256 X^{40}\\
  & - 195122200576 X^{42} +
  113311088640 X^{44} - 51748995072 X^{46}\\
  & + 17186013184 X^{48} -
  3000107008 X^{50} - 764936192 X^{52}\\
  & + 911474688 X^{54} -
  423231488 X^{56} + 119013376 X^{58} - 18874368 X^{60}\\
  & + 1048576 X^{62} +
  65536 X^{64}.
  \end{split}
\end{equation*}
The calculation is essentially complete, but we provide a few
additional comments. Despite the fact that \texttt{Mathematica}
could not solve this equation directly, it is possible to express
$d_1$ and $d_3$ in terms of radicals, as we demonstrate below.  It
is unfortunate that the formulas are prohibitively complicated.

We conclude by briefly describing how to recover explicit formulas
for $d_1$ and $d_3$.  First notice that if $f(x)=y$, with $f(x)$
defined in \eqref{f(x) definition}, then we can express $x$ in
terms of $y$ by solving quadratic equations. Since
$\dn(2K/5)=f(f(d_1))$ and $\dn(6K/5)=f(f(d_3))$, it is sufficient
to reduce $\dn(2K/5)$ and $\dn(6K/5)$ to radicals. There are
several methods to accomplish this calculation. The simplest
approach is to generate their minimal polynomials by repeated
applications of the duplication formula for $\dn(z)$. It is then
possible to verify the formulas
\begin{align*}
\dn\left(\frac{2K}{5}\right)=&\frac{1}{4}\left(1+\sqrt{5}+2\sqrt{2+\sqrt{5}}-\sqrt{2(5+\sqrt{5})}\right),\\
\dn\left(\frac{6K}{5}\right)=&\frac{1}{4}\left(1+\sqrt{5}-2\sqrt{2+\sqrt{5}}+\sqrt{2(5+\sqrt{5})}\right).
\end{align*}

In this example we have assumed that  $\alpha=1/2$. It is still
possible, albeit significantly more difficult, to evaluate these
elliptic functions for certain other values of $\alpha$.

\section{Conclusion}

We have shown how to prove many explicit formulas for $F_{(a,b)}(x)$.
 The most obvious extension of this research is to examine cases where $a>5$ is an odd integer.
   Notice that Theorem \ref{Jacobi Theorem} does not apply to those values.  It should also be
   interesting to attempt to apply our techniques to the class of sums studied by Zucker and
    McPhedran in \cite{ZM2}.  They gave closed form evaluations for many values of
\begin{equation*}
S(p,r,j)=\sum_{m=-\infty}^{\infty}\sum_{n=-\infty}^{\infty}
\frac{j^{2s}}{\left((j n+p)^2+(j m+r)^2\right)^s},
\end{equation*}
in terms of Dirichlet $L$-series.

\textit{Acknowledgements.} This paper arose from a problem
submitted to the \emph{American
  Mathematical Monthly} \cite{zm} and an associated unsolved problem also submitted to the
  \emph{Monthly}. We
thank Problems Editors Douglas Hensley, Kenneth Stolarsky, and
Douglas West for bringing these problems to our attention.  A
portion of this research was carried out while the third author
was visiting the University of Georgia. He is grateful for their
hospitality.  Finally we thank the referee for pointing out equation
\eqref{Fabx rapidly converging formula}, and for providing useful references.


\begin{thebibliography}{00}

\bibitem{AS}
M.~Abramowitz and I.~A.~Stegun, eds., Handbook of Mathematical
Functions, Dover, New York, 1965.

 \bibitem{geaI}
 G.~E.~Andrews and B.~C.~Berndt, Ramanujan's Lost Notebook, Part I, Springer,
 New York, 2005.

 \bibitem{III}
 B.~C.~Berndt, Ramanujan's Notebooks, Part III, Springer-Verlag, New York, 1991.


 \bibitem{V}
 B.~C.~Berndt, Ramanujan's Notebooks, Part V, Springer-Verlag, New York, 1998.


\bibitem{Bor} D.~Borwein, J.~M.~Borwein and K.~F.~Taylor,
``Convergence of lattice sums and Madelung's constant."
\textit{J.~Math.~Phys.}~\textbf{26}, no.~11 (1985): 2999--3009.

\bibitem{cox}
D.~Cox, Primes of the Form $x^2+ny^2$, Wiley, New York, 1989.


\bibitem{Cr} R.~E.~Crandall, ``New representations for the Madelung
constant." \textit{Experiment.~Math.}  \textbf{8}, no.~4 (1999):
367--379.


\bibitem{FG} P.~J.~Forrester and M.~L.~Glasser. ``Some new lattice
sums including an exact result for the electrostatic potential
within the NaCl lattice." \textit{J.~Phys.~A: Math.~Gen.}
\textbf{15} (1982): 911--914.


\bibitem{gr}
I.~S.~Gradshteyn and I.~M.~Ryzhik, eds., Table of Integrals,
Series, and Products, 5th ed., Academic Press, San Diego, 1994.



 \bibitem{RGlattice}
\textsc{M.~Rogers}, Hypergeometric formulas for lattice sums and
Mahler measures, \emph{Intern. Math. Res. Not.} (to appear),
preprint \texttt{arXiv:\,0806.3590 [math.NT]} (2008).

\bibitem{RZ}
\textsc{M.~Rogers} and \textsc{W.~Zudilin}, {}From $L$-series of
elliptic curves to Mahler measures, preprint
\texttt{arXiv:\,1012.3036 [math.NT]} (2010).

\bibitem{RZ2}
\textsc{M.~Rogers} and \textsc{W.~Zudilin}, {}On the Mahler
measure of $1+X+X^{-1}+Y+Y^{-1}$, preprint
\texttt{arXiv:\,1102.1153 [math.NT]} (2011).

\bibitem{ww}
E.~T.~Whittaker and G.~N.~Watson, A Course of Modern Analysis, 4th
ed., Cambridge University Press, Cambridge, 1966.

\bibitem{Zu} I.~J.~Zucker, ``Madelung constants and lattice sums for hexagonal
crystals."  \textit{J.~Phys.~A}~\textbf{24}, no.~4  (1991):
873--879.

\bibitem{zm}
I.~J.~Zucker and R.~McPhedran, ``Problem 11294,'' (with unpublished solutions by
the proposers, R.~Chapman, and A.~Stadler)  Amer.~Math.~Monthly \text{114}
(2007), 452.

\bibitem{ZM2}
I.~J.~Zucker and R.~McPhedran, ``Dirichlet $L$-series with real
and complex characters and their application to solving double
sums." \textit{Proc. R. Soc. A}~\textbf{464}, no.~2094 (2008):
1405--1422.

 \end{thebibliography}
\end{document}